\documentclass{amsart}
\usepackage{mathrsfs,epsfig} 

\theoremstyle{plain} \newtheorem{thm}{Theorem}[section]
\newtheorem{lemma}[thm]{Lemma} 
 \newtheorem{prop}[thm]{Proposition}

\theoremstyle{definition}

\numberwithin{equation}{section} \theoremstyle{remark}
\newtheorem{rem}{Remark}[section]

\newcommand{\dsp}{\displaystyle}
\newcommand\Q{\mathbb Q}\newcommand\R{\mathbb R}\newcommand\ff{\frac}
\newcommand\EE{\mathbb E}\newcommand\PP{\mathbb P}
\renewcommand\P{\mathbb P}
\newcommand\dd{\delta}\newcommand\DD{\Delta}
\newcommand\vv{\varepsilon}\newcommand\rr{\rho}
\renewcommand\gg{\gamma}

\newcommand\nn{\nabla}\newcommand\pp{\partial}
\renewcommand\d{\text{\rm{d}}} \newcommand\bb{\beta} \renewcommand\aa{\alpha}
 \newcommand\si{\sigma} \newcommand\F{\scr F}
\newcommand\Ric{\text{\rm{Ric}}} \newcommand\Hess{\text{\rm{Hess}}}
\newcommand\e{\text{\rm{e}}}

\newcommand\ee{\varepsilon}
\newcommand\newdot{{\kern.8pt\cdot\kern.8pt}}
\newcommand\tbull{{\hbox{$\displaystyle .$}}}
\newcommand{\scr}[1]{\mathscr #1}

\begin{document}

\title[Gradient Estimate and Harnack Inequality]{Gradient Estimate and Harnack Inequality on
     Non-Compact Riemannian Manifolds}

\thanks{Supported in part by
     NNSFC (10121101), RFDP (20040027009) and the  973-Project in China.}

\author[M. Arnaudon]{Marc Arnaudon} \address{D\'epartement de
  math\'ematiques\hfill\break\indent Universit\'e de Poitiers, T\'el\'eport 2
  - BP 30179\hfill\break\indent F--86962 Futuroscope Chasseneuil Cedex,
  France} \email{arnaudon@math.univ-poitiers.fr}

\author[A. Thalmaier]{Anton Thalmaier} \address{Institute of Mathematics,
  University of Luxembourg\hfill\break\indent 162A, avenue de la
  Fa\"{\i}encerie\hfill\break\indent L--1511 Luxembourg, Grand-Duchy of
  Luxembourg} \email{anton.thalmaier@uni.lu}

\author[F.-Y. Wang]{Feng-Yu Wang$^*$} \address{ 
               School of Mathematics, Beijing Normal University, Beijing 100875, China}
               \curraddr{WIMCS, Department of Mathematics, University of
               Wales Swansea\hfill\break\indent  
               Singleton Park, Swansea, SA2 8 PP, UK}
\thanks {$^*$ Corresponding author}
\email{wangfy@bnu.edu.cn}



\maketitle

\begin{abstract}
A new type of gradient estimate is established for diffusion
semigroups on non-compact complete Riemannian manifolds. As
applications, a global Harnack inequality with power and a heat
kernel estimate are derived for diffusion semigroups on arbitrary
complete Riemannian manifolds.
\end{abstract}

\section{The main result}
\label{sec:1}

Let $M$ be a  non-compact complete Riemannian manifold, and $P_t$
be the Dirichlet diffusion semigroup generated by $L= \DD+\nn V$
for some $C^2$ function $V$. We intend to establish reasonable
gradient estimates and Harnack type inequalities for~$P_t$.
In case that $\Ric-\Hess_V$ is bounded below, a dimension-free Harnack
inequality was established in \cite{W97}, which according to \cite{W04},
is indeed equivalent to the corresponding curvature condition.
See e.g.~\cite{B} for equivalent statements on heat kernel
functional inequalities; see also \cite{LY,BQ,Li} for a parabolic
Harnack inequality using the dimension-curvature condition by
shifting time, which goes back to the classical local parabolic
Harnack inequality of Moser \cite{M}.

Recently, some sharp gradient estimates have been derived in
\cite{SZ,Zhang} for the Dirichlet semigroup on relatively compact
domains. More precisely, for $V=0$ and a relatively compact
open $C^2$ domain $D$, the Dirichlet heat semigroup $P_t^D$ satisfies
\begin{equation}\label{1.1}
|\nn P_t^Df|(x)\le C(x,t)\,P_t^D f(x),\quad
 x\in D,\ t>0,
\end{equation}
for some locally bounded function $C\colon D\times
{]0,\infty[}\to{]0,\infty[}$ and all $f\in \scr B_b^+$, the space of
bounded non-negative measurable functions on $M$.
Obviously, this implies the Harnack inequality
\begin{equation}\label{1.2}
P_t^Df(x)\le \tilde C(x,y,t)\, P_t^D f(y),\quad
t>0,\ x,y\in D,\ f\in \scr B_b^+,
\end{equation}
for some function $\tilde C\colon M^2\times{]0,\infty[}\to {]0,\infty[}$.
The purpose of this paper is to establish inequalities analogous to
(\ref{1.1}) and (\ref{1.2}) globally on the whole manifold~$M$.

On the other hand however, both (\ref{1.1}) and (\ref{1.2}) are
in general wrong for $P_t$ in place of $P_t^D$.
A simple counter-example is already the standard heat semigroup on~$\R^d$.
Hence, we turn to search for the following slightly weaker version of gradient
estimate:
\begin{align}\label{1.3a}
|\nn P_t f(x)|\le \dd \big(P_t f\log f-
P_t f\log P_t f\big)(x)+ \ff{C(\dd, x)}{t\land 1}\,P_t f(x)&,\\
 x\in M,\ t>0,\ \dd>0,\
f\in\scr B_b^+&,\notag
\end{align}
for some positive function $C\colon
{]0,\infty[}\times M\to {]0,\infty[}$.
This kind of gradient estimate is
new and, in particular, implies the Harnack inequality with power
introduced in \cite{W97} (see Theorem \ref{T1.2} below).

\begin{thm}\label{T1.1} There exists a continuous positive function
$F$ on $]0,1]\times M$ such that
\begin{align}\label{1.3}
 |\nn P_t f(x)|&\le \dd \big(P_t f\log f-
P_t f\log P_t f\big)(x)\notag\\
&+\left(F(\dd\wedge 1,x) \left(\ff {1} {\dd (t\wedge 1)} +1\right)+\ff{2\dd}{\e}\right)P_t f(x),\\
&\dd>0, \ x\in M,\ t>0,\ f\in \scr B_b^+.\notag
\end{align}
\end{thm}

\begin{thm}\label{T1.2} There exists a positive function $C\in
C(]1,\infty[\times M^2)$ such that
\begin{align*}
(P_t f(x))^\aa\le (P_t f^\aa(y)) \exp \left\{\ff{2(\aa-1)}{\e}+\aa C(\aa,x,y)\left(\ff{
\aa\rr^2(x,y)}{(\aa-1)(t\wedge 1)}+\rr(x,y)\right)\right\}&,\\
\aa>1,\ t>0,\ x,y\in M,\ f\in \scr B_b^+&,
\end{align*}
where $\rr$ is the Riemannian distance on $M$.
Consequently, for any $\dd>2$ there exists a positive function
$C_\dd\in C({[0,\infty[}\times M)$ such that the transition density
$p_t(x,y)$ of $P_t$ with respect to $\mu(\d x):=\e^{V(x)}\d x$,
where $\d x$ is the volume measure, satisfies
$$p_t(x,y)\le\ff{\exp\left\{-{\rr(x,y)^2}/(2\dd
t)+C_\dd(t,x)+C_\dd(t,y)\right\}}{\sqrt{\mu(B(x,\sqrt {2t}))\mu(B(y,\sqrt
{2t}))}},\quad x,y\in M,\ t\in {]0,1[}\,.$$
\end{thm}

\begin{rem}
According to the Varadhan asymptotic formula for short time
behavior, one has
$\lim_{t\to 0} 4 t \log p_t(x,y)= -\rr(x,y)^2,\ x\ne y$.
Hence, the above heat kernel upper bound is sharp for short time.
\end{rem}

The paper is organized  as follows:
In Section~\ref{sec:2} we provide a formula expressing
$P_t$ in terms of $P_t^D$ and the joint distribution of $(\tau,
X_{\tau})$, where $X_t$ is the $L$-diffusion process and $\tau$ its
hitting time to $\pp D$.
Some necessary lemmas and technical results are collected.
Proposition \ref{Pr2.5} is a refinement of
a result in \cite{Zhang} to make the coefficient of $\rr(x,y)/t$
sharp and explicit.
In Section \ref{sec:3}
we use parallel coupling of diffusions together with Girsanov
transformation to obtain a gradient estimate for Dirichlet heat
semigroup. Finally, complete proofs of Theorems \ref{T1.1} and
\ref{T1.2} are presented in Section~\ref{sec:4}.

To prove the indicated theorems, besides stochastic arguments, we
make use of a local gradient estimate obtained in \cite{SZ} for
$V=0$. For the convenience of the reader, we include a brief proof
for the case with drift in the Appendix.

\section{Some Preparations}
\label{sec:2}

Let $X_s(x)$ be an $L$-diffusion process with starting point $x$ and
explosion time $\xi(x)$. For any open $C^2$ domain $D\subset M$ such that
$x\in D$, let $\tau(x)$ be the first hitting time of $X_s(x)$ at
the boundary $\pp D$.
We have
$$P_t f(x)= \mathbb E \left[f(X_t(x))\,1_{\{t<\xi(x)\}}\right],\quad P_t^D f(x)=
\mathbb E \left[f(X_t(x))\,1_{\{t<\tau(x)\}}\right].
$$ Let $p_t^D(x,y)$ be the
transition density of $P_t^D$ with respect to $\mu$.

We first provide a formula for the density $h_x(t,z)$ of $(\tau(x),
X_{\tau(x)}(x))$ with respect to $\d t\otimes\nu(\d z)$, where $\nu$
is the measure on $\pp D$ induced by $\mu(\d y):= \e^{V(y)}\d y$.

\begin{lemma} \label{L2.1} Let $K(z,x)$ be the Poisson kernel in
$D$ with respect to $\nu$. Then
\begin{equation}
\label{E13} h_x(t,z)=\int_D\left(-\pp_t
p^D_t(x,y)\right)K(z,y)\,\mu(\d y).
\end{equation} Consequently, the density $s\mapsto\ell_x(s)$ of $\tau(x)$ satisfies
the equation:
\begin{equation}
\label{E19} \ell_x(s)=\int_D\left(-\pp_t p^D_t(x,y)\right)\,\mu(\d
y).
\end{equation}
\end{lemma}

\begin{proof} Every bounded continuous function $f\colon\partial D\to\R$
extends continuously to a function $h$ on $\bar D$ which is harmonic
in $D$ and represented by
$$
h(x)=\int_{\partial D}K(z,x)f(z)\,\nu(\d z).
$$
Recall that $z\mapsto K(z,x)$ is the density of $X_{\tau(x)}(x)$.
Hence
$$
\EE[f(X_{\tau(x)}(x))]=h(x)=\int_{\partial D}K(z,x)f(z)\,\nu(\d z).
$$
On the other hand, the identity
$$
h(x)=\EE[h(X_{t\wedge \tau(x)})]
$$
yields
\begin{align*}
h(x)&=\int_Dp^D_t(x,y)h(y)\,\mu(\d y)+\int_{\partial D}\nu(\d
z)\int_0^th_x(s,z)f(z)\d s
\\
&=\int_Dp^D_t(x,y)\left(\int_{\partial D}K(z,y)f(z)\nu(\d
z)\right)\,\mu(\d y)+\int_{\partial D}\nu(\d z)\int_0^t
h_x(s,z)f(z)\d s
\\
&=\int_{\partial D}f(z)\left(\int_Dp^D_t(x,y)K(z,y)\,\mu(\d
y)+\int_0^t h_x(s,z)\d s\right)\nu(\d z),
\end{align*}
which implies that
\begin{equation}
\label{E12} K(z,x)=\int_Dp^D_t(x,y)K(z,y)\,\mu(\d y)+\int_0^t
h_x(s,z)\d s.
\end{equation}
Differentiating with respect to $t$ gives
\begin{equation}
\label{E11} h_x(t,z)=-\pp_t\int_Dp^D_t(x,y)K(z,y)\,\mu(\d y).
\end{equation}
Since $\pp_t p^D_t(x,y)$ is bounded on $[\vv,\vv^{-1}]\times \bar
D\times \bar D$ for any $\vv\in{]0,1[}\,$, Eq.~\eqref{E13} follows
by the dominated convergence.

Finally, Eq.~(\ref{E19}) is obtained by integrating \eqref{E13} with
respect to $\nu(\d z)$.
\end{proof}

\begin{lemma} \label{L2.2} The following formula holds:
\begin{equation*}\begin{split} P_t f(x)&= P_t^Df(x)+ \int_{{]0,t]}\times\pp D}
P_{t-s}f(z) h_x(s,z) \,\d s\nu(\d z)\\
& = P_t^Df(x)+ \int_{{]0,t]}\times\pp D} P_{t-s}f(z) P_{s/2}^D
h_\tbull (s/2,z)(x)\,\d s\nu(\d z).\end{split}\end{equation*}
\end{lemma}

\begin{proof} By the strong Markov property we have
\begin{equation}\label{2.1}\begin{split}
&P_tf(x)= \EE\left[f(X_t(x))1_{\{t<\xi(x)\}}\right]\\
& = \EE\left[f(X_t(x))1_{\{t<\tau(x)\}}\right]
+\EE\left[f(X_t(x))1_{\{\tau(x)<t<\xi(x)\}}\right]\\
&=P_t^Df(x) +\EE\Big[\EE \left[f(X_t(x))1_{\{\tau(x)<t<\xi(x)\}}
\vert (\tau(x),X_{\tau(x)}(x))\right]\Big] \\
&=P_t^Df(x)+\int_{{]0,t]}\times \partial
D}P_{t-s}f(z)h_x(s,z)\,ds\,\nu(\d z).
\end{split}\end{equation}
Next, since
\begin{align*}
\pp_s p^D_s(x,y) &= L p_s^D(\newdot, y)(x)=
 L P_{s/2}^Dp_{s/2}^D(\newdot, y)(x)\\
&= P_{s/2}^D(L p_{s/2}^D(\newdot, y))(x)=P_{s/2}^D(\pp_u p_{u}^D(\newdot, y)|_{u=s/2})(x),
\end{align*}
 it follows from (\ref{E13}) that
\begin{equation}
\label{**}
h_x(s,z) = P_{s/2}^D h_\tbull(s/2,z)(x).
\end{equation}
This completes the proof. 
\end{proof}

We remark that formula (\ref{**}) can also be derived from the
strong Markov property without invoking Eq.~(\ref{E13}). Indeed, for any
$u<s$ and any measurable set $A\subset \pp D$, the strong Markov
property implies that
\begin{equation*}\begin{split}
\mathbb P\left\{\tau(x)>s,\ X_{\tau(x)}(x)\in
 A\right\}&= \mathbb E\Big[\big(1_{\{u<\tau(x)\}}\,
\mathbb P\left\{\tau(x)>s,\ X_{\tau(x)}(x)\in
 A|\scr F_u\right\}\Big]\\
 &=\int_Dp_u^D(x,y)\,\mathbb P\left\{\tau(y)>s-u,\ X_{\tau(y)}(y)\in
 A\right\}\mu(\d y),
\end{split}\end{equation*}
and thus,
$$h_x(s,z) = P_{u}^D h_\tbull(s-u,z)(x),\quad
s>u>0,\ x\in D,\ z\in \pp D.$$

\begin{lemma}
\label{L2.3}
Let $D$ be a relatively compact open domain
and $\rr_{\pp D}$ be the Riemannian distance to the boundary $\pp D$.
Then there exists a constant $C>0$ depending on $D$ such that
$$\mathbb P\{\tau(x)\le t\}\le C \e^{-\rr_{\pp D}^2(x)/16t},\quad
x\in D,\ t>0.$$
\end{lemma}

\begin{proof} For $x\in D$, let $R:= \rr_{\pp D} (x)$ and $\rr_x$ the Riemannian
distance function to $x$. Since $D$ is relatively compact, there
exists a constant $c>0$ such that $L \rr_x^2\le c$ holds on $D$
outside the cut-locus of $x$. Let $\gg_t:= \rr_x(X_t(x)),\ t\ge
0$. By It\^o's formula, according to Kendall \cite{K}, there exists a
 one-dimensional Brownian motion $b_t$ such that
$$\d \gg_t^2 \le 2\sqrt 2 \gg_t\,\d b_t + c\,\d t,\quad t\le \tau(x).$$
Thus, for fixed $t>0$ and $\dd>0$,
$$Z_s:= \exp\left(\ff \dd t \gg_s^2 -\ff \dd t cs - 4 \ff{\dd^2}
{t^2} \int_0^s \gg_u^2\d u\right),\quad s\le \tau(x)$$ is a
supermartingale. Therefore,
\begin{equation*}\begin{split}
\mathbb P\{\tau(x)\le t\} &= \mathbb
P\left\{\max_{s\in [0,t]} \gg_{s\land\tau(x)}\ge R\right\}
\le \mathbb P\left\{
\max_{s\in [0,t]} Z_{s\land \tau(x)} \ge \e^{\dd R^2/t -\dd c-
4\dd^2R^2/t}\right\}\\
&\le \exp\left( c\dd - \ff 1 t (\dd R^2-4\dd^2
R^2)\right).\end{split}\end{equation*}
The proof is completed by
taking $\dd:= 1/8$.
\end{proof}

\begin{lemma}\label{L2.5}
On a measurable space
$(E,\F,\tilde\mu)$ satisfying $\tilde\mu(E)<\infty$, let $f\in L^1(\tilde\mu)$ be non-negative with
$\tilde\mu(f)>0$. Then for every measurable function $\psi$ such that
 $\psi f\in L^1(\tilde\mu)$, there holds:
\begin{equation}
\label{E2} \int_E \psi f\,\d\tilde\mu \le \int_E f\log\ff f
{\tilde\mu(f)}\,\d\tilde\mu+\tilde\mu(f) \log\int_E\e^\psi \,\d\tilde\mu.
\end{equation}
\end{lemma}

\begin{proof} This is a direct consequence of \cite{S} Lemma~6.45. We give a proof for completeness.
Multiplying $f$ by a positive constant, we can assume that $\tilde\mu(f)=1$. 
If $\int_E\e^\psi \,\d\tilde\mu=\infty$, then \eqref{E2} is clearly satisfied. 

If $\int_E\e^\psi \,\d\tilde\mu<\infty$, then since $\int_E\e^\psi \,\d\tilde\mu\ge \int_{\{f>0\}}\e^\psi \,\d\tilde\mu$, we can assume that $f>0$ everywhere. Now from the fact that $e^\psi\frac{1}{f}\in L^1(f\tilde\mu)$,  we can apply Jensen's inequality to obtain
$$
\log\left(\int_E e^\psi \, \d\tilde\mu\right)=\log\left(\int_E e^\psi\frac{1}{f} \, f\d\tilde\mu\right)
\ge \int_E\log\left(e^\psi\frac{1}{f}\right)\,f\d\tilde\mu
$$
(note the right-hand-side belongs to $\R\cup\{-\infty\}$).
To finish we remark that since $\psi f\in L^1(\tilde\mu)$,
$$
\int_E\log\left(e^\psi\frac{1}{f}\right)\,f\d\tilde\mu=\int_E \psi f\,d\tilde\mu-\int_E f\log f\,d\tilde\mu.
$$
\end{proof}

Finally, in order to obtain precise gradient estimate of the type
(\ref{1.3}),
where the constant in front of $\rr(x,y)/t$ is explicit and sharp,
we establish the following revision of \cite[Theorem 2.1]{Zhang}.

\begin{prop} \label{Pr2.5}
Let $D$ be a relatively compact open $C^2$
domain in $M$ and $K$ a compact subset of $D$.
For any $\vv>0$,
there exists a constant $C(\vv)>0$ such that
\begin{align}
|\nn\log p_t^D(\newdot,y)(x)|\le \ff{C(\vv)\log
(1+t^{-1})}{\sqrt t} + \ff{(1+\vv)\rr(x,y)}{2t}&,\notag \\
t\in {]0,1[},\ x\in K,\ y\in D&.
\label{A21}
\end{align}
In addition, if $D$ is convex, the
above estimate holds for $\vv=0$ and some constant $C(0)>0$.
\end{prop}

\begin{proof} Since $\dd:=\min_K \rr_{\pp D}>0$, it suffices to deal with the
case where $0<t\le 1\land \dd$. To this end, we combine the
argument in \cite{Zhang} with relevant results from
\cite{W98,W07}.\smallskip

(a) \ Let $t_0=t/2$ and $y\in D$ be fixed. Take
$$f(x,s)= p_{s+t_0}^D(x,y),\quad x\in D,\ s>0.$$
Applying Theorem \ref{TA1} of the Appendix to the cube $$Q:=
B(x,\rr_{\pp D}(x))\times [s- \rr_{\pp D}(x)^2/2, s]\subset D\times
[-t_0,t_0], \quad s\leq t_0,$$ we obtain
\begin{equation}\label{AA1}|\nn\log f(x,s)|\le \ff {c_0}{\rr_{\pp
D}(x)}\Big(1+ \log \ff A {f(x,s)}\Big),\quad s\le t_0,\end{equation}
where $A:=\sup_Q f$ and $c_0>0$ is a constant depending on the
dimension and curvature on $D$. By \cite[Theorem 5.2]{Li},
\begin{equation}\label{AA2}A\le c_1 f\left(x, s+ \rr_{\pp D}(x)^2\right),\quad
s\in {]0,1]},\ x\in D,\end{equation} holds for some constant $c_1>0$
depending on $D$ and $L$. Moreover, by the boundary Harnack
inequality of \cite{FGS} (which treats $Z=0$ but generalizes 
easily to non-zero $C^1$ drift~$Z$), \begin{equation}\label{AA3} 
f\left(x,s+\rr_{\pp D}(x)^2\right)\le c_2 f(x,s),\quad s\in {]0,1]},\ x\in
D,\end{equation} for some constant $c_2>0$ depending on $D$ and $L$.
Combining (\ref{AA1}), (\ref{AA2}) and (\ref{AA3}), there exists a
constant $c>0$ depending on $D$ and $L$ such that
\begin{equation}\label{A23} \left|\nn\log f(x,s)\right|\le \ff c {\sqrt s},
\quad x\in D,\ s\in{]0,t_0]}\ \text{ with } \rr_{\pp D}(x)^2\le s.
\end{equation}

(b) \ Let $$
 \Omega= \left\{(x,s):\ x\in D,\ s\in [0,t_0],\ \rr_{\pp D}(x)^2\ge s\right\}$$
and $B= \sup_{\Omega} f$. Since $\pp_s f= Lf$, for any constant
$b\ge 1$,
we have
$$(L-\pp_s)\Big(f\log \ff{bB}f\Big) =-\ff {|\nn f|^2}f.$$
Next, again by $\pp_s f=Lf$ and the Bochner-Weizenb\"ock formula,
$$(L-\pp_s) \ff{|\nn f|^2}f\ge - 2 k \ff{|\nn f|^2} f,$$ where $k\ge
0$ is such that $\Ric-\nn Z\ge -k$ on $D$. Then the function
$$h:= \ff {s|\nn f|^2}{(1+2k s)f} - f\log \ff {bB}f$$
satisfies
\begin{equation}\label{A20}(L-\pp_s)h\ge 0\quad \text{on}\ D\times
{]0,\infty[}.\end{equation}
Obviously $h(\newdot, 0)\le 0$, and
(\ref{A23}) yields $h(x,s)\le 0$ for $s=\rr_{\pp D}(x)^2$ provided the
constant $b$ is large enough. Then the maximum principle and
inequality (\ref{A20}) imply $h\le 0$ on~$\Omega$. Thus,
\begin{equation}\label{A24} |\nn\log f(x,s)|^2 \le (2k+ s^{-1}) \log \ff {bB}f,\
\quad (x,s)\in \Omega.\end{equation}

(c) \ If $D$ is convex, by \cite[Theorem 2.1]{W98} with $\dd=\sqrt
t$ and $t=2t_0$, we obtain (note the generator therein is $\ff 1 2
L$)
$$f(x,t_0)= p_{2t_0}^D(x,y)=p_{2t_0}^D(y,x)\ge c_1\varphi(y) \,t_0^{-d/2}
\e^{-\rr(x,y)^2/8t_0},\quad x\in K,\ y\in D$$
for some constant $c_1>0$, where $\varphi>0$ is the first Dirichlet
eigenfunction of $L$ on~$D$.
On the other hand, the intrinsic ultracontractivity for
$P_t^D$ implies (see e.g.~\cite{OW})
$$f(z,s)= p_{s+t_0}^D(z,y)\le c_2 \,\varphi(y)\, t_0^{-(d+2)/2}, \quad
z,y\in D,\ s\le t_0,$$ for some constant $c_2>0$ depending on $D$, $K$
and $L$. Combining these estimates we obtain
$$\ff B {f (x,s)}\le c_3\, t_0^{-1} \e^{\rr(x,y)^2/8t_0},\quad 
x\in K,\ s\le t_0,$$ for some
constant $c_3>0$ depending on $D$, $K$ and $L$. Hence by (\ref{A24}) for
$s=t_0$ we get the existence of a constant $C>0$ such that
$$|\nn \log p_{2t_0}^D(\newdot,y)|^2 \le (t_0^{-1}+ 2k) \left(C+
\log t_0^{-1} + \ff{\rr(x,y)^2}{8t_0}\right)$$ for all $y\in D$, 
$x\in K$ and $t_0\in {]0,1[}$ with $t_0\le
\rr_{\pp D}(x)^2$. This completes the proof by noting that $t=
2t_0$.\smallskip

(d) \ Finally, if $D$ is not convex, then there exists a constant
$\si>0$ such that
$$\langle\nn_NX,X\rangle\ge -\si|X|^2, \quad \ X\in T\pp D,$$ where $N$ is the
outward unit normal vector field of $\pp D$.
Let $f\in C^\infty(\bar D)$ such that $f=1$ for
$\rr_{\pp D}\ge \vv,\ 1\le f\le
\e^{2\vv\si}$ for $\rr_{\pp D}\le \vv$, and $N\log f|_{\pp D}\ge
\si$. By Lemma~2.1 in~\cite{W07}, $\pp D$ is convex under the metric
$\tilde g:= f^{-2}\langle\newdot,\newdot\rangle$.
Let $\tilde\DD, \tilde\nn$ and $\tilde\rr$
be respectively the Laplacian, the gradient and the Riemannian
distance induced by $\tilde g$.
By Lemma~2.2 in~\cite{W07},
$$L:=\DD +\nn V= f^{-2}\left[\tilde\DD+(d-2)f\nn f\right] +\nn V.$$
Since $D$ is convex under $\tilde g$, as explained in the first
paragraph in Section 2 of \cite{W07},
$$\tilde g(\tilde\nabla \tilde\rho(y,\newdot), \tilde\nn\varphi)|_{\pp D}<0,$$
so that
$$\tilde\si(y):= \sup_D \tilde g(\tilde\nabla \tilde\rho(y,\newdot), \tilde\nn\varphi)
<\infty,\quad y\in D.$$
Hence, repeating the proof of Theorem 2.1 in
\cite{W98}, but using $\tilde \rr$ and $\tilde\nn$ in place of $\rr$ and $\nn$
respectively, and taking into account that $f\to 1$ uniformly as $\vv\to 0$, we
obtain
\begin{equation*}\begin{split} p^D_{2t_0}(x,y)&\ge C_1(\vv)\varphi(y)
t_0^{-d/2}\e^{-C_2(\vv)\tilde \rr(x,y)^2/8t_0}\\
&\ge C_1(\vv)\varphi(y) t_0^{-d/2}\e^{-C_2(\vv)C_3(\vv)
\rr(x,y)^2/8t_0}\end{split}\end{equation*} for some constants
$C_1(\vv), C_2(\vv), C_3(\vv)>1$ with $C_2(\vv), C_3(\vv) \to 1$ as
$\vv\to 0$. Hence the proof is completed.
\end{proof}

\section{Gradient estimate for Dirichlet heat semigroup
using coupling of diffusion processes}
\label{sec:3}

\begin{prop}
\label{P3.1} Let $D$ be a relatively compact $C^2$ domain in $M$.
For every compact subset $K$ of $D$, there exists
a constant $C=C(K,D)>0$ such that for all $\delta>0$, $t>0$, $x_0\in K$
and for all bounded positive functions $f$ on~$M$,
\begin{equation}
\label{E7}
\begin{split} &|\nn P_t^Df(x_0) |\\&\le \delta
P_t^D\left(f\log\left(\frac{f}{P_t^Df(x_0)}\right)\right)(x_0)
+C\left(\frac{1}{\delta (t\wedge 1)}+1\right)P_t^Df(x_0).
\end{split}
\end{equation}
\end{prop}

\begin{proof}  
We assume that $t\in]0,1[$, the other case will be treated at the very end of the proof.

We write $\nabla
V=Z$ so that $L=\Delta +Z$. Since $P_t^D$ only depends on the
Riemannian metric and the vector field $Z$ on the domain $D$,
by modifying the metric and $Z$
outside of $D$ we may assume that $\Ric-\nn Z$ is bounded below (see
e.g. \cite{TW}); that is,
\begin{equation}\label{**0} \Ric-\nn Z\ge -\kappa\end{equation} for some
constant $\kappa\ge 0$.

Fix $x_0\in K$. Let $f$ be a positive bounded function on $M$ and $X_s$ a
diffusion with generator $L$, starting at $x_0$.
For fixed $t\le 1$, let
$$
v=\frac{\nn P_t^Df(x_0)}{|\nn P_t^Df(x_0)|}
$$
and denote by $u\mapsto \varphi(u)$ the geodesics in $M$ satisfying $\dot
\varphi(0)=v$. Then
$$
\left.\frac{\d}{\d u}\right\vert_{u=0}P_t^Df(\varphi(u))=\big|\nn P_t^Df(x_0)\big|.
$$
To formulate the coupling used in \cite{ATW}, we introduce some
notations.

If $Y$ is a semimartingale in $M$, we denote by $\d Y$ its It\^o
differential and by $\d_m Y$ the martingale part of $\d Y$:
in local coordinates,
$$
\d Y=\left(\d Y^i+\frac12\Gamma_{jk}^i(Y)\,
\d\langle Y^j,Y^k\rangle\right)\frac{\partial}{\partial x^i}
$$
where $\Gamma_{jk}^i$ are the Christoffel symbols of the Levi-Civita
connection; if $\d Y^i=\d M^i+\d A^i$ where $M^i$ is a local
martingale and $A^i$ a finite variation process, then
$$
\d_mY=\d M^i\frac{\partial}{\partial x^i}.
$$
Alternatively, if $\dsp Q(Y)\colon T_{Y_0}M\to T_{Y_\tbull}M$ is
the parallel translation along $Y$, then
$$
\d Y_t=Q(Y)_{t}\,\d\left(\int_0^{\raise1pt\hbox{$\displaystyle .$}} Q(Y)_{s}^{-1}
\circ \d Y_s\right)_t
$$
and
$$
\d_mY_t=Q(Y)_{t}\,\d N_t
$$
where $N_t$ is the martingale part of the Stratonovich integral
$\int_0^t Q(Y)_{s}^{-1} \circ dY_s$.

For $x,y\in M$, and $y$ not in the cut-locus of $x$, let
\begin{equation}
\label{E15}
 I(x,y)=\sum_{i=1}^{d-1} \int_0^{\rho(x,y)}\left(|\nabla_{\dot
e(x,y)}J_i|^2+\big\langle R(\dot e(x,y),J_i)J_i+\nabla_{\dot
e(x,y)}Z,\dot e(x,y)\big\rangle\right)_s\,\d s
\end{equation}
where $\dot e(x,y)$ is the tangent vector of the unit speed minimal
geodesic $e(x,y)$ and $(J_i)_{i=1}^d$ are Jacobi fields along
$e(x,y)$ which together with $\dot e(x,y)$ constitute an
orthonormal basis of the tangent space at $x$ and $y$:
$$
J_i(\rho(x,y))=P_{x,y}J_i(0),\quad i=1,\ldots,d-1;
$$
here $P_{x,y} \colon T_{x}M\to T_{y}M$ is the parallel translation along
the geodesic $e(x,y)$.

Let $c\in {]0,1[}$.
For $h>0$ but smaller than the injectivity radius of $D$, and $t>0$,
let $X^h$ be the semimartingale satisfying
$X_0^h=\varphi(h)$ and
\begin{equation}
\label{E14} \d X_s^h=P_{X_s,X_s^h}\d_mX_s+Z(X^h_s)\,\d s+ \xi_s^h\d
s,\end{equation} where
$$\xi_s^h:=\left(\frac{h}{ct}+\kappa h\right)n(X_s^h,X_s)
$$
with $n(X_s^h,X_s)$ the derivative at time $0$ of the unit speed
geodesic from $X_s^h$ to $X_s$, and $P_{X_s, X_s^h}\colon T_{X_s}M\to
T_{X_s^h}M$ the parallel transport along the minimal geodesic
from $X_s$ to $X_s^h$. By convention, we put $n(x,x)=0$ and
$P_{x,x}=\text{Id}$ for all $x\in M$.

By the second variational formula and (\ref{**0}) (cf.~\cite{ATW}),
we have
$$\d\rr(X_s, X_s^h)\le \left\{I(X_s, X_s^h) - \ff {h}{ct}-\kappa
h\right\}\d s\le \ -\ff{h}{ct}\,\d s,\quad s\le T_h,$$
where $T_h:=\inf\{s\ge 0: X_s=X_s^h\}$.
Thus, $(X_s, X_s^h)$ never reaches the cut-locus.
In particular, $T_h\le ct$ and
\begin{equation}\label{3.4*}
X_s= X_s^h,\quad s\ge ct.\end{equation}
Moreover,
we have $\rr(X_s, X_s^h)\le h$ and 
\begin{equation}\label{*L}
|\xi_s^h|^2 \le h^2 \Big(\kappa + \ff 1
{ct}\Big)^2.\end{equation}
We want to compensate the additional drift of $X^h$ by a change of probability.
To this end, let
$$
M_s^h=-\int_0^{s\wedge ct}\big\langle \xi_r^h,
P_{X_r,X_r^h}\,\d_mX_r\big\rangle,
$$
 and
$$
R_s^h=\exp\left(M_s^h-\frac12[M^h]_s\right).
$$
Clearly $R^h$ is a martingale, and under $\Q^h=R^h\cdot \P$,
the process $X^h$
is a diffusion with generator~$L$.

Letting $\tau(x_0)$ (resp. $\tau^h$) be the hitting time of
$\partial D$ by $X$ (resp. by $X^h$), we have
$$
1_{\{t<\tau^h\}}\le 1_{\{t<\tau(x_0)\}}+1_{\{\tau(x_0)\le
  t<\tau^h\}}.
$$
But, since $X_s^h=X_s$ for $s\ge ct$, we obtain
$$
1_{\{\tau(x_0)\le
  t<\tau^h\}}=1_{\{\tau(x_0)\le
  ct\}}1_{\{
  t<\tau^h\}}.
$$
 Consequently,
\begin{align*}
\frac{1}{h}\left(P_t^Df(\varphi(h))-P_t^Df(x_0)\right)
&=\frac{1}{h}\,\EE\left[f(X^h_t)R_t^h1_{\{t<\tau^h\}}-f(X_t(0))1_{\{t<\tau(x_0)\}}\right]\\
&\le \frac{1}{h}\,\EE\left[f(X^h_t)R_t^h1_{\{t<\tau(x_0)\}}-f(X_t(0))1_{\{t<\tau(x_0)\}}\right]\\
&\qquad+\frac{1}{h}\,\EE\left[f(X^h_t)R_t^h1_{\{\tau(x_0)\le
  ct\}}1_{\{
  t<\tau^h\}}\right],
\end{align*}
and since $X_t^h=X_t$ this yields
\begin{align}
\label{E4}
\begin{split}
\frac{1}{h}\left(P_t^Df(\varphi(h))-P_t^Df(x_0)\right)
&\le\EE\left[f(X_t)1_{\{t<\tau(x_0)\}}\frac{1}{h}(R_t^h-1)\right]\\
&\qquad +\frac{1}{h}\,\EE\left[f(X^h_t)R_t^h1_{\{\tau(x_0)\le ct\}}
1_{\{t<\tau^h\}}\right].
\end{split}
\end{align}

The left hand side converges to the quantity to be evaluated as $h$ goes
to $0$. Hence, it is enough to find appropriate $\limsup$'s for the two
terms of the right hand side. We begin with the first term. Letting
$$
Y_s^h=\left|M_s^h -\frac12[M^h]_s\right|
$$
and noting that
$\langle n(X_r^h, X_r), P_{X_r, X_r^h}\d_m X_r\rangle=\sqrt 2\,\d b_r$
up to the coupling time $T_h$ for some one-dimensional Brownian
motion $b_r$, we have
\begin{align*}
R_t^h&=\exp\left(M_t^h-\frac12[M^h]_t\right) \le
1+M_t^h-\frac12[M^h]_t+(Y_t^h)^2\exp(Y_t^h)\\
&=1+M_t^h-\int_0^t |\xi_s^h|^2\d s +(Y_t^h)^2\exp(Y_t^h).
\end{align*}   From the assumptions,
$\exp(Y_t^h)$ and $Y_t^h/h$
have all their moments bounded, uniformly in $h>0$. Consequently,
since $f$ is bounded,
\begin{align*}
\limsup_{h\to0}\EE\left[f(X_t)1_{\{t<\tau(x_0)\}}\frac{1}{h}\left(\int_0^t|\xi_r^h|^2\,dr
+(Y_t^h)^2\exp(Y_t^h)\right)\right]=0,
\end{align*}
which implies
\begin{align*}
\limsup_{h\to 0}\EE&\left[f(X_t)1_{\{t<\tau(x_0)\}}\frac{1}{h}(R_t^h-1)\right]\\
&\le\limsup_{h\to0}\EE\left[f(X_t)1_{\{t<\tau(x_0)\}}\frac1h\int_0^s\big\langle
\xi_r^h,P_{X_r, X_r^h}\,\d_m X_r\big\rangle\right].
\end{align*}
Using Lemma \ref{L2.5} and estimate (\ref{*L}), we have for $\delta>0$
\begin{align*}
\EE&\left[f(X_t)1_{\{t<\tau(x_0)\}}\frac1h\int_0^s\big\langle
\xi_r^h,P_{X_r, X_r^h}\d_m X_r\big\rangle\right]\\
&\le \delta
P_t^D\left(f\log\left(\frac{f}{P_t^Df(x_0)}\right)\right)(x_0)\\
&\quad+\delta
P_t^Df(x_0)\log\EE\left[1_{\{t<\tau(x_0)\}}\exp\left(\frac{1}{\delta h
}\int_0^{ct}\big\langle
   \xi_s^h, P_{X_s, X_s^h}\d_mX_s\big\rangle\right)\right]\\
&\le \delta
P_t^D\left(f\log\left(\frac{f}{P_t^Df(x_0)}\right)\right)(x_0)\\
&\quad+\delta
P_t^Df(x_0)\log\EE\left[\exp\left(\frac{1}{\delta^2h^2}\int_0^{ct}\left|
   \xi_s^h\right|^2 \,\d s\right)\right]\\
&\le \delta
P_t^D\left(f\log\left(\frac{f}{P_t^Df(x_0)}\right)\right)(x_0)
+\delta P_t^Df(x_0)\frac{ct}{\delta^2}\left(\frac{1}{c^2t^2}+\kappa^2\right)\\
&\le \delta
P_t^D\left(f\log\left(\frac{f}{P_t^Df(x_0)}\right)\right)(x_0)
+\frac{C'}{c\delta t}P_t^Df(x_0),
\end{align*}
where $\displaystyle C'=1+(c\kappa)^2$ (recall that $t\le 1$).
Since the last expression is independent of~$h$, this proves that
\begin{align}
\limsup_{h\to0}\EE&\left[f(X_t)1_{\{t<\tau(x_0)\}}\frac{1}{h}(R_t^h-1)\right]\notag \\
&\le\delta
P_t^D\left(f\log\left(\frac{f}{P_t^Df(x_0)}\right)\right)(x_0)
+\frac{C'}{c\delta t}P_t^Df(x_0).
\label{E21}
\end{align}
We are now going to estimate $\limsup$ of the second term in~\eqref{E4}.
By the strong Markov property, we
have
\begin{align}
\EE\left[f(X^h_t)R_t^h1_{\{\tau(x_0)\le
   ct\}}1_{\{t<\tau^h\}}\right]
&=\EE_{\Q^h}\left[P_{t-ct}^Df(X^h_{ct})1_{\{\tau(x_0)\le
   ct<\tau^h\}}\right]\notag \\
&\le \|P_{t-ct}^Df\|^{\mathstrut}_\infty \,\Q^h\big\{\tau(x_0)\le
   ct<\tau^h\big\}.
\label{E9}
\end{align}
Since $\displaystyle\rho(X_s^h,X_s)\le
h\frac{ct-s}{ct}$ for $s\in [0,ct]$, we have on $\{\tau(x_0)\le ct<\tau^h\}$:
$$\rho_{\partial D}(X^h_{\tau(x_0)})\le h\frac{ct-\tau(x_0)}{ct}.$$
For $s\in [0,\tau^h-\tau(x_0)]$, define
$$
Y'_s=\rho(X_{\tau(x_0)+s}^h,\partial D),
$$
and for fixed small $\ee>0$ (but $\ee>h$), let
$S'=\inf\{s\ge 0,\
Y'_s=\ee\ \text{ or }\ Y_s'=0\}$. Since under $\mathbb Q^h$ the
process $X_s^h$ is generated by $L$, the drift of
$\rr(X_s^h, \pp D)$ is $L\rr(\newdot,\pp D)$ which is bounded in a neighborhood of
$\pp D$.
Thus, for a sufficiently small $\vv>0$,
there exists a $\Q^h$-Brownian motion $\beta$ started at $0$, and a constant
 $N>0$ such that
$$Y_s:=h\,\frac{ct-\tau(x_0)}{ct}+\sqrt 2 \beta_s+Ns\ge Y_s',\quad s\in [0,S'].$$
Let
$$
S=\inf\big\{u\ge 0,\ Y_u=\ee\ \text{ or }\  Y_u=0\big\}.
$$
Taking into account that on $\{\tau(x_0)=u\}$,
$$
\{Y'_{S'}=\ee\}\cup\{S'>ct-u\}\subset \{Y_S=\ee\}\cup\{S>ct-u\},
$$
we have for $u\in [0,ct]$,
\begin{align*}
\Q^h\big\{ct<\tau^h\vert \tau(x_0)=u\big\}&\le
\Q^h\big\{Y_{S'}=\ee\vert \tau(x_0)=u\big\}
+\Q^h\big\{S'\ge ct-u\vert \tau(x_0)=u\big\}\\
&\le
\Q^h\big\{Y_{S}=\ee\vert \tau(x_0)=u\big\}
+\Q^h\big\{S\ge ct-u\vert \tau(x_0)=u\big\}\\
&\le \Q^h\big\{Y_{S}=\ee\vert \tau(x_0)
=u\big\}+\frac{1}{ct-u}\EE_{\Q^h}\big[S\vert \tau(x_0)=u\big].
\end{align*}
Now using the fact that $\e^{-NY_s}$ is a martingale and $Y_s^2-2s$
a submartingale, we get
$$
\Q^h\left\{Y_{S}=\ee\vert
\tau(x_0)=u\right\}=\frac{1-\e^{-Nh\frac{ct-u}{ct}}}{1-\e^{-N\ee}}\le C_1h
$$
and
\begin{align*}
\EE_{\Q^h}\big[S\vert \tau(x_0)=u\big]&\le \EE_{\Q^h}\big[Y_S^2\vert
  \tau(x_0)=u\big]\\
&\le \ee^2\,\Q^h\big\{Y_{S}=\ee\vert
  \tau(x_0)=u\big\}\\
&=\ee^2\frac{1-\e^{-Nh\frac{ct-u}{ct}}}{1-\e^{-N\ee}}\le C_2\,\frac{h(ct-u)}{ct}
\end{align*} for some constants $C_1, C_2>0$.
Thus,
\begin{align*}
\Q^h\big\{ct<\tau^h\vert \tau(x_0)=u\big\}
&\le C_1h+\frac{1}{ct-u}\,C_2\,\frac{h(ct-u)}{ct}\\
&\le C_1h+C_3\frac{h}{ct}\le C_4\frac{h}{t}
\end{align*}
for some constants $C_3, C_4>0$ (recall that $t\le 1$).
Denoting by $\ell^h$ the density of $\tau(x_0)$ under
$\Q^h$, this implies
\begin{align*}
\Q^h\big\{\tau(x_0)\le
    ct<\tau^h\big\}&=\int_0^{ct}\ell^h(u)\,\Q^h\{ct<\tau^h\vert\si^h=u\}\,\d u\\
&\le C_4\frac{h}{t }\int_0^{ct}\ell^h(u)\, \d u\\
&= C_4\frac{h}{t }\,\Q^h\big\{\tau(x_0)\le ct\big\}.
\end{align*}
In terms of $D^{-h}=\{x\in D,\ \rho_{\partial D}(x)>h\}$ and
$\sigma^h=\inf\{s>0,\ X_s^h\in\partial D^{-h}\}$, we have
$\sigma^h\le \tau(x_0)$ a.s. Hence, by Lemma \ref{L2.3},
$$
\Q^h\big\{\tau(x_0)\le ct\big\}
\le \Q^h\big\{\sigma^h\le ct\big\}
\le C\exp\left\{-\frac{\rho_{\partial D^{-h}}(\varphi(h))}{16ct}\right\},
$$
where we used that $X^h_s$ is generated by $L$ under $\Q^h$.
This implies
\begin{equation}
\label{E10}
\Q^h\left\{\tau(x_0)\le ct<\tau^h\right\}
\le C_5\frac{h}{t}\exp\left\{-\frac{\rho_{\partial D^{-h}}(\varphi(h))}{16ct}\right\}.
\end{equation}
Since $\displaystyle
\frac{1}{h}\left(P_t^D(\varphi(h))-P_t^D(x_0)\right)$ converges to
$|\nn P_t^Df(x_0)|$, we obtain from~\eqref{E4}, \eqref{E21},
\eqref{E9} and \eqref{E10},
\begin{align}
|\nn P_t^Df(x_0)|&\le \delta
P_t^D\left(f\log\left(\frac{f}{P_t^Df(x_0)}\right)\right)(x_0)\notag\\
&\quad
+\frac{C'}{c\delta t}\,P_t^Df(x_0)+C_5\,\|P_{t-ct}^Df\|_\infty^{\mathstrut}\frac{1}{t}
\exp\left\{{-\frac{\rho_{\partial D}(x_0)}{16ct}}\right\}.
\label{E5}
\end{align}
Finally, as explained in steps c) and d) of the proof of
Proposition \ref{Pr2.5}, for any compact set $K\subset D$, there exists a constant
$C(K,D)>0$ such that
$$ \|P_{t-ct}^Df\|_\infty^{\mathstrut}\le
\e^{C(K,D)/t}P_t^Df(x_0),\quad c\in [0,1/2],\ x_0\in K,\ t\in {]0,1}].
$$
Combining this with (\ref{E5}), we arrive at
\begin{align}
\label{E6}
\begin{split}
|\nn P_t^Df(x_0)|&\le \delta
P_t^D\left(f\log\left(\frac{f}{P_t^Df(x_0)}\right)\right)(x_0)
+\frac{C'}{c\delta t}P_t^Df(x_0)\\
&\quad+C_5 \frac{1}{t}\exp\left\{-\frac{\rho_{\partial D}(x_0)}{16ct}\right\}\,
\exp\left\{\frac{C(K,D)} t\right\}P_t^Df(x_0).
\end{split}
\end{align}
Finally, choosing $c$ such that
$$0<c< \ff 1 2\land \ff{\text{dist}(K,\pp D)}{16 C(K,D)},$$
we get for some
constant $C>0$,
\begin{align}
\label{E22}
 |\nn P_t^Df(x_0)|\le \delta
P_t^D\left(f\log\left(\frac{f}{P_t^Df(x_0)}\right)\right)(x_0)
+C\left(\frac{1}{\delta t}+1\right)\,P_t^Df(x_0),&\\ x_0\in K,\ \dd>0,&\notag
\end{align} 
which implies the desired inequality.

To finish we consider the case $t> 1$. From the semigroup property, we have $P_t^Df=P_1^D(P_{t-1}^Df)$. So letting $g=P_{t-1}^Df$ and applying~\eqref{E22} to $g$ at time~$1$, we obtain 
$$
|\nn P_t^Df(x_0)|\le \delta
P_1^D\left(g\log\left(\frac{g}{P_1^Dg(x_0)}\right)\right)(x_0)
+C\left(\frac{1}{\delta }+1\right)\,P_1^Dg(x_0).
$$ 
Now using $P_1^Dg=P_t^Df$, we get 
$$
|\nn P_t^Df(x_0)|\le \delta
P_1^D(g\log g)(x_0)-P_t^Df(x_0)\log P_t^Df(x_0)
+C\left(\frac{1}{\delta }+1\right)\,P_t^Df(x_0).
$$ 
Letting $\varphi(x)=x\log x$, we have for $z\in D$
\begin{align*}
g\log g(z)&=\varphi \left(\EE\left[f(X_{t-1}(z))1_{\{t-1<\tau(z)\}}\right]\right)\\
&\le \EE\left[\varphi\left(f(X_{t-1}(z))1_{\{t-1<\tau(z)\}}\right)\right]\\
&=\EE\left[\varphi(f)(X_{t-1}(z))1_{\{t-1<\tau(z)\}}\right]\\
&=P_{t-1}^D(f\log f)(z),
\end{align*}
where we successively used the convexity of $\varphi$ and the fact that  $\varphi (0)=0$. This implies 
$$
|\nn P_t^Df(x_0)|\le \delta
P_t^D\left(f\log\left(\frac{f}{P_t^Df(x_0)}\right)\right)(x_0)
+C\left(\frac{1}{\delta }+1\right)\,P_t^Df(x_0),
$$
which is the desired inequality for $t>1$.
\end{proof}

\section{Proof of Theorems \ref{T1.1} and Theorem \ref{T1.2} }
\label{sec:4}

\begin{proof}[of Theorem \ref{T1.1}\/] \ 
 We assume that $t\in ]0,1[$ and refer to the end of the proof of Proposition~\ref{P3.1} for the case $t>1$. 
Fixing $\dd>0$ and $x_0\in
M$, we take $R= 160/(\dd\wedge1)$. Let $D$ be a relatively compact
open domain with $C^2$ boundary containing $B(x_0,2R)$ and contained
in $B(x_0, 2R+\ee)$ for some small $\ee>0$.
By the countable compactness of
$M$, it suffices to prove that there exists a constant $C=C(D)$ such
that (\ref{1.3}) holds on $B(x_0,R)$ with $C$ in place of $F(\dd\wedge 1,x_0)$. We now fix $x\in B(x_0,R)$, $t\in
{]0,1]}$ and $f\in \scr B_b^+$. Without loss of generality, we may and
will assume that $P_t f(x)=1$.\smallskip

(a) Let $P_s(x,\d y)$ be the transition kernel of the $L$-diffusion
process, and for $x\in D$, $z\in M$, let
$$\nu_s(x,\d z)= \int_{\pp D}h_x(s/2,y)\, P_{t-s}(y,\d z)\,\nu(\d y),$$
where $\nu$
is the measure on $\pp D$ induced by $\mu(\d y)= \e^{V(y)}\d y$.
By Lemma \ref{L2.2} we have
$$P_t f(x)= P_t^D f(x) + \int_{{]0,t]}\times
D\times M} p_{s/2}^D(x,y)\, f(z)\, \d s \mu(\d y)\nu_s(y,\d z) .$$
Then
\begin{align}
 |\nn P_t f(x)|&\le |\nn P_t^D f(x)|\notag \\&
\ \ \ +\int_{{]0,t]}\times D\times M}
|\nn\log p_{s/2}^D(\newdot,y)(x)|\,p_{s/2}^D(x,y)f(z)\,\d s\mu(dy)\nu_s(y,\d z)\notag \\
&=: \ I_1+I_2.\label{B1}
\end{align}

(b) By Proposition \ref{P3.1},  we have \begin{equation}\label{B3} I_1\le \dd
P_t^D (f\log f)(x)+\ff{\dd}{\e} +C\left(\ff{1}{\dd t}+1\right),\quad x\in B(x_0, R),\
t\in ]0,1[,\ \dd>0\end{equation} for some $C=C(D)>0$.\smallskip

(c) By Proposition \ref{Pr2.5} with $\vv=1$, we have
\begin{equation}\label{B4} I_2\le\int_{{]0,t]}\times M\times D}
\Big[\ff{C\log(\e +s^{-1})}{\sqrt s }+\ff {2\rr(x,y)} {s}\Big]
p_{s/2}^D(x,y)\,f(z)\,\d s \nu_s(y, \d z)\mu(\d y)\end{equation} for
some $C=C(D)>0$ and all $t\in {]0,1]}$.
Applying Lemma \ref{L2.5} to
the measure $\tilde\mu:=p^D_{s/2}(x,y)\,\d s\,\nu_s(y,\d z)\mu(\d y)$ on
$E:={]0,t]}\times M\times D$ so that
$$\tilde\mu(E)= \PP(\tau(x)\le t<\xi(x)) \le  1,$$ we obtain
\begin{align}
&I_2 \le \dd\,\EE\left[ (f\log
f)(X_t(x))1_{\{\tau(x)\le t<\xi(x)\}}\right]+\ff{\dd}{\e} 
+\dd \EE\left[f(X_t(x))1_{\{\tau(x)\le t<\xi(x)\}}\right]
\notag\\
& \quad\times\log\int_{{]0,t]}\times M\times D}
\exp\left\{\ff{C\log(\e +s^{-1})}{\dd\sqrt s}
+\ff{2\rr(x,y)}{s\dd}\right\}\d s\,p_{s/2}^D(x,y)\nu_s(y,\d z)\,\mu(\d y)\notag\\
&\quad\le \dd \EE\left[ (f\log
f)(X_t(x))1_{\{\tau(x)\le t<\xi(x)\}}\right] +\ff{\dd}{\e} 
+\dd \EE\left[f(X_t(x))1_{\{\tau(x)\le t<\xi(x)\}}\right]\notag\\
&\quad\qquad \times\log\int_{{]0,t]}\times M\times D}
\exp\left\{\ff{A}{\dd}+\ff{9R}{s\dd}\right\} \d s\, p_{s/2}^D(x,y)\nu_s(y,\d z)\,\mu(\d y),
\label{B5}\end{align}
where $$ A:=\sup_{r>0}\big\{C\sqrt r\log(\e+r)-r\big\}<\infty. $$
We get
 \begin{align}
I_2&\le 
\dd \EE\left[ (f\log
f)(X_t(x))1_{\{\tau(x)\le t<\xi(x)\}}\right] +\ff{\dd}{\e}\notag\\
&\qquad + \dd \EE\left[f(X_t(x))1_{\{\tau(x)\le t<\xi(x)\}}\right]\left(\log\EE\big[
\exp\left(9R/\dd\tau(x)\right)\big]+\ff{A}{\dd}\right)
\notag \\
&\le \dd \EE\left[ (f\log
f)(X_t(x))1_{\{\tau(x)\le t<\xi(x)\}}\right] +\ff{\dd}{\e}+ \dd\log\EE\big[
\exp\left(9R/\dd\tau(x)\right)\big]+A\notag \\
& \le \dd \EE\left[ (f\log
f)(X_t(x))1_{\{\tau(x)\le t<\xi(x)\}}\right] \notag \\
&\qquad+ \dd\log\EE\left[
\exp\left(\ff{9R}{(\dd\wedge1)\tau(x)}\right)^{\ff{\dd\wedge1}{\dd}}\right]+A+\ff{\dd}{\e}\notag \\
&=\dd \EE\left[ (f\log f)(X_t(x))1_{\{\tau(x)\le t<\xi(x)\}}\right] \notag \\
&\qquad+ (\dd\wedge1)\log\EE\left[ \exp\left(\ff{9R}{(\dd\wedge1)\tau(x)}\right)\right]+A+\ff{\dd}{\e}
 .\label{B6} \end{align}
By Lemma
\ref{L2.3} and noting that $\rr_{\pp}(x)\ge R$, we have
\begin{equation*}\begin{split} \EE&\left[
\exp\left(\ff{9R}{(\dd\wedge1)\tau(x)}\right)\right]\le
1+\EE\left[\ff{9R}{(\dd\wedge1)\tau(x)}
\exp\left(\ff{9R}{(\dd\wedge1)\tau(x)}\right)\right]\\
&=1+\int_0^\infty
\ff{9Rs}{(\dd\wedge1)}\exp\left(\ff{9Rs}{(\dd\wedge1)}\right)\ff{\d}{\d
s}
\left(-\P\{\tau(x)\le s^{-1}\}\right)\,\d s\\
&= 1+\ff{9R}{(\dd\wedge1)}\int_0^\infty
\left(\ff{9R}{(\dd\wedge1)}s+1\right)\exp\left(\ff{9Rs}{(\dd\wedge1)}\right)
\P\{\tau(x)\le s^{-1}\}\,\d s\\
&\le 1+\ff{9R}{(\dd\wedge1)}\int_0^\infty
\left(\ff{9R}{(\dd\wedge1)}s+1\right)\exp\left(\ff{9Rs}{(\dd\wedge1)}\right)
\exp\left(\ff{-R^2s}{16}\right)\,\d s\\
&=1+\ff{9R}{(\dd\wedge1)}\int_0^\infty
\left(\ff{9R}{(\dd\wedge1)}s+1\right)\exp\left(\ff{-Rs}{(\dd\wedge1)}\right)\,\d s\\
&=1+9\int_0^\infty \left(9u+1\right)\exp\left(-u\right)\,du=:A',
\end{split}\end{equation*}
since $R=160/(\dd\wedge1)$. This along with (\ref{B6}) yields
\begin{equation}
\label{newlabel}
I_2\le\dd\,\EE\left[ (f\log f)(X_t(x))1_{\{\tau(x)\le t<\xi(x)\}}\right] + \log
A'+A+\ff{\dd}{\e}\,.
\end{equation}
The  proof is completed by
combining (\ref{newlabel}) with (\ref{B1}), (\ref{B3}) and (\ref{B5}).
\end{proof}

\begin{proof}[of Theorem \ref{T1.2}\/] \ By Theorem
\ref{T1.1},
\begin{align}
 |\nn P_t f(x)|&\le \dd \big(P_t
(f\log f)(x) - (P_t f)(x)\log P_t f(x)\big)\notag\\
&\quad + \left(F(\dd\wedge 1,x)\left(\ff {1} {\dd (t \wedge 1)} +1\right)+\ff{2\dd}{e}\right) P_t f(x),\quad
\dd>0,\ x\in M.
\label{A30} \end{align}
For $\aa>1$ and $x\ne y$, let $\bb(s)=
1+ s(\aa-1)$ and let $\gg\colon [0,1]\to M$ be the minimal geodesic from $x$
to $y$. Then $|\dot \gg|=\rr(x,y)$.
Applying (\ref{A30}) with $\dd=
\ff{\aa-1}{\aa \rr(x,y)}$, we obtain
\begin{equation*}\begin{split} &\ff{\d}{\d s} \log (P_t
f^{\bb(s)})^{\aa/\bb(s)}(\gg_s) \\
&= \ff{\aa (\aa-1)}{\bb(s)^2}\,\ff{P_t (f^{\bb(s)}\log
f^{\bb(s)}) -(P_t f^{\bb(s)})\log P_t f^{\bb(s)}}{P_t
f^{\bb(s)}}(\gg_s) \\ &\qquad + \ff\aa {\bb(s)}\, \ff{\langle\nn P_t
f^{\bb(s)},\dot\gg_s\rangle}{P_t f^{\bb(s)}}(\gg_s)\\
&\ge \ff{\aa \rr(x,y)}{\bb(s) P_t f^{\bb(s)}(\gg_s)} \bigg\{
\ff{\aa-1}{\aa\rr(x,y)} \Big(P_t (f^{\bb(s)}\log f^{\bb(s)}) -(P_t
f^{\bb(s)})\log P_t f^{\bb(s)}\Big)(\gg_s)\\
&\qquad- |\nn P_t f^{\bb(s)}(\gg_s)|\bigg\}\\
&\ge
-F\left(\ff{\aa-1}{\aa\rr(x,y)}\wedge 1,\gg_s\right)\left(\ff{\aa^2\rr^2(x,y)}{\beta(s)(\aa-1)(t\wedge 1)}
+\ff{\aa\rr(x,y)}{\bb(s)}\right)-\ff{2(\aa-1)}{\e\bb(s)}\\&\ge
-C(\aa,x,y)\left(\ff{\aa\rr^2(x,y)}{(\aa-1)(t\wedge 1)}+\rr(x,y)\right)-\ff{2(\aa-1)}{\e}
\end{split}
\end{equation*}
where
$C(\aa,x,y):= \sup_{s\in [0,1]}\ff1{\aa}F\left(\ff{\aa-1}{\aa\rr(x,y)}\wedge 1,\gg_s\right)$. This implies the
desired Harnack inequality.

Next, for fixed $\aa\in ]1,2[$, 
let
$$K(\aa,t,x)= \sup\big\{C(\aa, x,y):\ y\in B(x,\sqrt {2 t})\big\},\quad t>0,\
x\in M.$$
Note $K(\aa,t,x)$ is finite and continuous in $(\aa,t,x)\in ]1,2[\times ]0,1[\times M$.
Let $p:=2/\aa$.
For fixed $t\in ]0,1[$,  the Harnack inequality  gives for $y\in B(x,\sqrt{2t})$,
$$(P_t f(x))^2\le (P_t f^\aa(y))^p \exp\left\{\ff{2(2-p)}{\e}+ 2 K(\aa,t,x) \left(\ff
{2\aa}{\aa-1} +\sqrt{2t}\right)\right\}.$$
Then choosing $T>t$ such that $q:=
p/2(p-1)< T/t$,
\begin{equation*}\begin{split}  & \mu\big(B(x,\sqrt{2
t})\big)\exp\left\{-\ff{2(2-p)}{\e}-2 K(\aa,t,x) \left(\ff{2\aa}{\aa-1}+\sqrt{ 2 t}\right)
-\ff{t}{T-qt}\right\} (P_t f(x))^2\\
&\le \int_{B(x,\sqrt{2t})} (P_t f^\aa(y))^p
\exp\left\{-\ff{\rr(x,y)^2}{2(T-qt)}\right\}\mu(\d
y).\end{split}\end{equation*}
Similarly to the proof of
\cite[Corollary 3]{ATW}, we obtain that for any $\dd>2$, choosing $\aa=\ff{2\dd}{2+\dd}\in]1,2[$ such that  $ \dd>\ff{2}{2-\aa}=\ff{p}{p-1}>2$,
 there is a constant $c(\dd)>0$ such that the following estimate holds:
\begin{equation*}\begin{split} E_\dd(x,t)&:= \int_M p_t(x,y)^2
\exp\left\{\ff{\rr(x,y)^2}{\dd t}\right\}\mu(\d y)\\
&\le \ff{\exp\left\{c(\dd)K(\aa,t,x)(1+\sqrt{2t})\right\}}{\mu(B(x,\sqrt{2t})},\
\quad t>0,\ x\in M.
\end{split}\end{equation*}
By \cite[Eq. (3.4)]{G}, this implies the desired heat kernel
upper bound for $C_\dd(t,x):= c(\dd)K(\aa,t,x)(1+\sqrt{2 t})$.
\end{proof}

\section{Appendix}
\label{sec:5}

The aim of the Appendix is to explain that the arguments in
Souplet-Zhang \cite{SZ} and Zhang \cite{Zhang} for gradient
estimates of solutions to heat equations work as well in the case with
drift.

\begin{thm} \label{TA1} Let $L= \DD +Z$ for a $C^1$ vector field $Z$.
Fix $x_0\in M$ and $R,\ T,\ t_0>0$ such that $B(x_0,R)\subset M$. Assume
that
\begin{equation}\label{A0}\Ric - \nn Z\ge -K\end{equation}
on $B(x_0,R)$.
There exists a constant $c$ depending only on $d$, the dimension of the manifold,
such that for any positive solution $u$ of
\begin{equation}\label{A1} \pp_t u= Lu \end{equation} on $Q_{R,T} :=
B(x_0,R)\times [t_0-T, t_0]$, the estimate
$$|\nn \log u|\le c \Big(\ff 1 R +T^{-1/2}+\sqrt K\Big) \Big(1+ \log
\ff{\sup_{Q_{R,T}}u} u\Big)$$ holds on $Q_{R/2, T/2}$.
\end{thm}

\begin{proof} Without loss of generality, let $N:=\sup_{Q_{T,R}}u=1$;
otherwise replace $u$ by $u/N$.
Let $f= \log u$ and
$\omega=\ff{|\nn f|^2}{(1-f)^2}$. By (\ref{A1}) we have
$$Lf +|\nn f|^2 -\pp_t f=0$$ so that
\begin{equation}\label{A2} \begin{split} \pp_t \omega &= \ff{2\langle\nn f,\nn \pp_t
f\rangle}{(1-f)^2} + \ff{2\,|\nn f|^2\pp_t f}{(1-f)^3}\\
&= \ff{2\langle\nn f,\nn (L f +|\nn f|^2)\rangle}{(1-f)^2} 
+ \ff{2\,|\nn f|^2(Lf +|\nn f|^2)}{(1-f)^3}\\
&= \ff{2\langle\nn f,\nn (\DD f +|\nn f|^2)\rangle}{(1-f)^2} 
+ \ff{2\,|\nn f|^2(\DD f +|\nn f|^2)}{(1-f)^3}\\
&\qquad +\ff{2\langle\nn_{\nn f}Z,\nn f\rangle+ 2 \Hess_f (\nn
f,Z)}{(1-f)^2} + \ff{2\,|\nn f|^2\langle Z,\nn
f\rangle}{(1-f)^3}.\end{split}\end{equation} Moreover,
\begin{equation}\label{A3} \begin{split} L\omega &=\DD \omega + \ff{\langle Z,\nn
|f|^2\rangle}{(1-f)^2} + \ff{2 |\nn f|^2 \langle Z, \nn f\rangle}{(1-f)^3}\\
&=\DD\omega + \ff{2\,\Hess_f(\nn f,Z)}{(1-f)^2} +\ff{2\,|\nn
f|^2\langle Z,\nn f\rangle}{(1-f)^3}.\end{split}\end{equation}
Finally, by the proof of \cite[(2.9)]{SZ} with $-k$ replaced by
$\Ric (\nn f,\nn f)/|\nn f|^2$, we obtain
\begin{equation}\label{A4}
\begin{split} \DD\omega &-\left\{\ff{2\,\langle\nn f,\nn (\DD f
+|\nn f|^2)\rangle}{(1-f)^2} + \ff{2\,|\nn f|^2(\DD f +|\nn
f|^2)}{(1-f)^3}\right\}\\
&\ge \ff{2f}{1-f}\langle\nn f,\nn\omega\rangle
+2(1-f)\omega^2 +\ff{2\omega\, \Ric(\nn f,\nn f)}{|\nn f|^2}.
\end{split}\end{equation}
Combining (\ref{A0}),
(\ref{A2}), (\ref{A3}) and (\ref{A4}), we arrive at
$$L\omega-\pp_t\omega
\ge \ff{2f}{1-f}\langle\nn f,\nn\omega\rangle + 2 (1-f)\omega^2-2K\omega.$$
This implies the desired estimate by the Li-Yau cut-off
argument as in \cite{SZ}; the only difference is, using the notation
in \cite{SZ}, in the calculation of $-(\DD\psi)\omega $
after Eq.~(2.13) in \cite{SZ}. By (\ref{A0}) and the generalized Laplacian
comparison theorem (see \cite[Theorem 4.2]{BQ}), we have
$$L r \le \sqrt{Kd} \coth \big(\sqrt{K/d}\, r\big) \le \ff d r
+\sqrt{Kd},$$
and then
$$ -(L\psi)\omega = - (\pp_r^2 \psi +(\pp_r
\psi) Lr)\omega \le \left(|\pp_r \psi|^2 + |\pp_r \psi|\frac dr
+\sqrt{Kd}\,|\pp_r \psi|\right)\omega.$$
The remainder of the proof is the same as
in the proof of \cite[Theorem 1.1]{SZ},
using $L\psi$ in place of $\DD\psi$.
\end{proof}

\end{document}